\documentclass[11pt, reqno, draft]{amsart}
\usepackage{amsmath}
\usepackage{amstext}
\usepackage{amsbsy}
\usepackage{amsopn}
\usepackage{upref}
\usepackage{amsthm}
\usepackage{amsfonts}
\usepackage{amssymb}
\usepackage{mathrsfs}
\usepackage{times} 
\allowdisplaybreaks
 \newtheorem{theorem}{Theorem}
 
 \newtheorem{proposition}[theorem]{Proposition}
 
\theoremstyle{definition}
 
\theoremstyle{remark}
 \newtheorem{remark}[theorem]{Remark}

\newcommand{\p}{\partial}
\newcommand{\RR}{\mathbb{R}}
\newcommand{\TT}{\mathbb{T}}
\begin{document}
\title[Dispersive Flow]{A Remark on the 
global existence of a third order 
dispersive flow 
into locally Hermitian symmetric spaces
}
\author[E.~Onodera]{Eiji Onodera}
\address[Eiji Onodera]{Faculty of Mathematics, Kyushu University, Fukuoka-city,
812-8581, Japan}
\email{onodera@math.kyushu-u.ac.jp}
\begin{abstract}
We prove global existence of solutions 
to the initial value problem 
for a third order dispersive flow 
into compact locally Hermitian symmetric spaces. 
The equation we consider generalizes 
two-sphere-valued completely integrable 
systems modelling the motion of vortex filament.
Unlike one-dimensional Schr\"odinger maps, 
our third order equation is not 
completely integrable under the curvature condition 
on the target manifold in general. 
The idea of our proof is to exploit  
two conservation laws and an energy 
which is not necessarily preserved in time 
but does not blow up in finite time.
\end{abstract}
\maketitle
\section{Introduction}
\label{section:introduction}
Let $(N,J,g)$ be a compact almost Hermitian manifold 
with an almost complex structure $J$ and a Hermitian metric $g$. 
Let $\nabla$ be the Levi-Civita connection with respect to $g$. 
Consider the initial value problem(IVP) for a third order 
dispersive partial differential equation of the form 
\begin{align}
  u_t
  =
  a\,\nabla_x^2u_x
  +
  J_u\nabla_xu_x
  +
  b\,g(u_x,u_x)u_x
& \quad\text{in}\quad
  \mathbb{R}{\times} X,
\label{equation:pde}
\\
  u(0,x)
  =
  u_0(x)
& \quad\text{in}\quad
  X,
\label{equation:data}
\end{align}
where 
$u$ is an unknown mapping of $\mathbb{R}\times X$ to $N$, 
$(t,x)\in\mathbb{R}\times X$, 
$X$ denotes $\RR$ or 
$\mathbb{T}(=\mathbb{R}/\mathbb{Z})$, 
$u_t=du(\p/\p{t})$, 
$u_x=du(\p/\p{x})$, 
$du$ is the differential of the mapping $u$, 
$u_0$ is a given initial curve on $N$, and
$a,b\in\mathbb{R}$ are constant. 
$u(t)$ is a curve on $N$ for fixed $t\in\mathbb{R}$, 
and $u$ describes the motion of a curve subject to 
\eqref{equation:pde}. 
$\nabla_x$
is the covariant derivative
induced from $\nabla$ in the direction $x$ 
along the mapping $u$,
and 
$J_u$ denotes 
the almost complex structure at $u{\in}N$. 
\par 
The equation \eqref{equation:pde} geometrically generalizes
two-sphere-valued completely integrable systems 
which model the motion of vortex filament.  
In \cite{darios}, 
Da Rios first formulated the motion of vortex filament as 
\begin{equation}
\vec{u}_t=\vec{u}\times\vec{u}_{xx},
\label{equation:darios}
\end{equation}
where 
$\vec{u}=(u^1,u^2,u^3)$ is an $\mathbb{S}^2$-valued function of $(t,x)$, 
$\mathbb{S}^2$ is a unit sphere in $\mathbb{R}^3$ 
with a center at the origin, 
and $\times$ is the exterior product in $\mathbb{R}^3$. 
The physical meanings of $\vec{u}$ and $x$ are 
the tangent vector and the signed arc length 
of vortex filament respectively. 
When $a,b=0$, \eqref{equation:pde} generalizes \eqref{equation:darios} 
and solutions to \eqref{equation:pde} 
are called one-dimensional Schr\"odinger maps. 
In \cite{FM}, 
Fukumoto and Miyazaki proposed 
a modified model equation of
vortex filament 
\begin{equation}
\vec{u}_t=\vec{u}\times\vec{u}_{xx}
+
a
\left[
\vec{u}_{xxx}
+
\frac{3}{2}
\left\{\vec{u}_x\times(\vec{u}\times\vec{u}_x)\right\}_x
\right].
\label{equation:FM}
\end{equation}
When $b=a/2$, \eqref{equation:pde} generalizes \eqref{equation:FM}. 
We call solutions to
\eqref{equation:pde} dispersive flows.
\par 
In recent ten years, 
the generalized form \eqref{equation:pde} 
has been studied 
in order to understand the relation between 
the structure of  
\eqref{equation:pde} as a partial differential equation 
and the geometric setting for $N$. 
In this article, having same motivation in mind, 
we are concerned with 
the existence (and the uniqueness) of solutions 
to the IVP for \eqref{equation:pde}-\eqref{equation:data}. 
\par 
For $\mathbb{S}^2$-valued physical models such as 
\eqref{equation:darios} and \eqref{equation:FM},  
time-local and global existence theorem is well studied.  
More precisely, 
Sulem, Sulem and Bardos proved time-local and global 
existence of a unique
solution to the IVP
for \eqref{equation:darios} in \cite{SSB}. 
Nishiyama and Tani showed time-local and global 
existence theorem for 
\eqref{equation:FM} in \cite{NT} and \cite{TN}. 
In their results, 
some conservation laws of the equation 
played the crucial parts.
\par 
Restricting to the case of K\"ahler manifolds 
as $N$, 
short-time existence results for
\eqref{equation:pde}-\eqref{equation:data} 
have already been well established. 
Roughly speaking, 
the K\"ahler condition 
$\nabla J\equiv 0$
ensures that the equation behaves as  
symmetric hyperbolic systems  
and hence the mix of the 
classical energy method and geometric 
analysis works to their proof. 
When $a,b=0$, 
Koiso showed the short-time existence of a unique solution 
in the class $H^{m+1}(\mathbb{T};N)$ for any integer $m\ge 1$.   
See \cite{koiso}
(see \cite{PWW} if $X=\RR$). 
His work 
was pioneering in the sense that 
the $L^2$-based bundle-valued Sobolev space 
$H^m(X;TN)$ for $u_x$
was revealed to be suitable to understand the structure 
of the equation for the first time. 
After that, short-time existence results 
for higher-dimensional Schr\"odinger maps 
were established. 
See, \cite{Ding}, 
\cite{MCGAHAGAN} 
and references therein.
When $a\ne 0, b\in \RR$, 
the author showed the short-time existence of a unique solution 
in the class $H^{m+1}(X;N)$ for any integer $m\ge 2$
(see \cite{onodera1}). 
\par
If $\nabla J\not\equiv 0$, 
a loss of one derivative occurs in the equation and 
the classical energy method does not work well. 
However, very recently,  
Chihara succeeded to prove short-time existence theorem 
for higher-dimensional Schr\"odinger maps
without assuming the K\"ahler condition in \cite{chihara}. 
Also for the third order equation \eqref{equation:pde}, 
he and the author showed short-time existence theorem 
when $a\ne 0$ and $b\in \RR$ 
without assuming the K\"ahler condition. 
See \cite{CO} and \cite{onodera3}. 
The idea of their proof is to construct 
a gauge transformation on the pull-back bundle 
$u^{-1}TN$ to eliminate the seemingly bad first 
order derivative loss.  
These results require more regularity $m\ge 4$ 
for the class of the solution. 
\par 
On the other hands, 
global existence results for \eqref{equation:pde}-\eqref{equation:data}
have been studied by adding some  
more conditions on $N$. 
When $a,b =0$ and $X=\mathbb{T}$, 
Koiso proved that 
the solution exists globally in time 
if the K\"ahler manifold $N$ 
is the locally Hermitian symmetric space
($\nabla R\equiv 0$) by finding a  
conservation law in \cite{koiso}. 
Pang, Wang and Wang obtained the same results 
when $a,b=0$ and $X=\RR$ in \cite{PWW}. 
Being inspired with Hasimoto's 
pioneering work in \cite{hasimoto}, 
Chang, Shatah and Uhlenbeck 
constructed a good moving frame along the map  
and rigorously reduced the equation 
for the one-dimensional Schr\"odinger map
to a simple form of a complex-valued 
nonlinear Schr\"odinger equation
to
discuss the global existence of
the Schr\"odinger map into Riemann surfaces. 
Though their argument is restricted only to the case 
where $X=\RR$ and the map is assumed to have 
a fixed point on $N$
as $x\to -\infty$, 
this reduction gives us understandings on an essential 
structure of one-dimensional Schr\"odinger maps. 
(see \cite{CSU}). 
For the case  
$a\ne 0$, 
the author proved the global existence theorem 
by assuming that 
$N$ is the compact Riemann surface with constant 
Gaussian curvature $K$ and $b=aK/2$ in \cite{onodera1}. 
Under the condition, \eqref{equation:pde} behaves as 
completely integrable systems 
and some conservation laws of the equation 
work in the proof. 
However, without such assumption, 
\eqref{equation:pde} cannot be expected to be 
completely integrable in general, 
even if $\nabla R\equiv 0$ is assumed 
as in the case $a,b=0$.
\par 
The aim of this article is to establish 
a global existence theorem for 
\eqref{equation:pde}-\eqref{equation:data} 
under the condition $\nabla R\equiv 0$
also when $a\ne 0$,
without the previous assumption in
\cite{onodera1}. 
The main theorem is the following:
\begin{theorem}
\label{theorem:main} 
Let $(N,J,g)$ be a compact locally Hermitian symmetric space, 
$a\ne 0, b\in \RR$, and let 
$m$ be a positive integer satisfying $m\geqslant 2$. 
Then, for any $u_0{\in}H^{m+1}(X;N)$,  
the initial value problem 
{\rm \eqref{equation:pde}-\eqref{equation:data}} 
admits a unique solution 
$u{\in}C(\RR;H^{m+1}(X;N))$.  
\end{theorem}
Theorem~\ref{theorem:main} gives not only 
an extension of the previous result by the author 
in \cite{onodera1} for the case $a\ne 0$ 
but also an analogue of the result by Koiso in \cite{koiso} 
for the case $a,b=0$.
\par 
To prove the theorem, 
we apply two conservation laws 
and an energy quantity for this equation.
More precisely, 
we use the following integral quantities of the form 
\begin{align}
E_1(u)
&=
a\, \|\nabla_xu_x\|_{L^2}^2
-\frac{b}{2}
\int_X
\left( 
g(u_x,u_x)
\right)^2
dx
-\int_X
g(u_x, J\nabla_xu_x)
dx, 
\label{equation:energy1}
\\
E_2(u)
&=
3a\, \|\nabla_x^2u_x\|_{L^2}^2
-10b\, 
\int_X
\left( 
g(u_x,\nabla_xu_x)
\right)^2
dx
\nonumber 
\\
&
\quad
-5b\, 
\int_X
g(u_x,u_x)
g(\nabla_xu_x, \nabla_xu_x)
dx
+2a\, 
\int_X
g(R(u_x,\nabla_xu_x)u_x,\nabla_xu_x)
dx.
\label{equation:energy2}
\end{align}
While $\|u_x(t)\|_{L^2}^2$ and 
$E_1(u(t))$ are preserved in time,  
$E_2(u(t))$ is not necessarily preserved in time.  
However, 
the a priori estimate itself 
for $\|\nabla_x^2u_x(t)\|_{L^2}^2$ 
can be obtained by careful computation. 
They imply a bound for $u_x(t)$ in $H^2(X;TN)$, which, in view of
the local existence result, prevents the formation of a finite-time
singularity.
\par 
The idea of finding such quantities comes from 
\cite{Laurey} and \cite{onodera2}.  
To explain this,
assume that $N$ is a compact Riemann surface with constant 
Gaussian curvature $K$ and $X=\RR$. 
From \cite[Theorem~1]{onodera2},  the equation 
\eqref{equation:pde} 
for 
$u(t,x):\RR_t\times \RR_x\to N$ which has a 
fixed point on $N$ as $x\to -\infty$
can be reduced to a third order  
dispersive equation with constant coefficient of the form 
\begin{align}
& 
q_t -aq_{xxx}-\sqrt{-1}q_{xx}
  =
  \left(
  \frac{a}{2}K+2b\right)
  \lvert{q}\rvert^{2}
  q_{x}
  -
  \left(
  \frac{a}{2}K-b
  \right)
  q^{2}\bar{q}_{x} 
  +
  \frac{\sqrt{-1}}{2}
  K
  \lvert{q}\rvert^{2}q
\label{equation:t3rd}
\end{align}
for complex-valued function 
$q(t,x):\RR_t\times \RR_x\to \mathbb{C}$. 
This reduction is obtained via the  
relation 
\begin{equation} 
u_x=q_1 e+q_2J e, 
\quad 
q=q_1+\sqrt{-1}q_2, 
\quad 
\nabla_x e=0,
\label{eq:frame0}
\end{equation}
where $\{e, J e \}$ is the moving frame 
along $u$ 
introduced by Chang, Shatah and 
Uhlenbeck in \cite{CSU}. 
On the other hands, 
the global existence theorem
for the equation of the form 
\begin{align}
q_t+A q_{xxx}
-\sqrt{-1}B q_{xx}
&=
-\sqrt{-1}\alpha |q|^2q
+\beta |q|^2_xq
+\gamma |q|^2q_x
\label{equation:333}
\end{align}
was established in the class 
$H^2(X;\mathbb{C})$ by Laurey in \cite{Laurey},  
where 
$A$, 
$B$, 
$\alpha$, 
$\beta$, 
$\gamma \in \RR$   
and $A\ne 0$, $\beta \ne 0$. 
The key idea of her proof was to exposit nice quantities of the form
\begin{align}
&-3A \beta \|q_x\|_{2}^2
-\beta 
\left(\beta + \frac{\gamma}{2} \right)
\|q\|_{4}^2
+
\sqrt{-1}\, 
\left(
B(2\beta +\gamma )
-3A \alpha
\right)
\int_X
q\bar{q}_x
dx, 
\label{equation:ttest1}
\\
&
3A \, 
\|q_{xx}\|_{2}^2
+
(6\beta +4\gamma)
\int_X
|q|^2|q_x|^{2}
dx
+
(4\beta +\gamma)
\operatorname{Re}
\int_X
q^2\bar{q}_x^2
dx, 
\label{equation:ttest2}
\end{align}
and $\|q\|_{2}^2$, 
where $\|\cdot\|_{p}$ is the standard $L^p$-norm for complex-valued 
function on $X$.
See (5.8), (5.12) and (5.4) respectively in \cite{Laurey}. 
If we set 
\begin{equation}
A=-a, \
B=1, \
\alpha =-K/2, \
\beta =b-aK/2,\
\gamma =b+aK
\label{equation:111}
\end{equation}
and 
take 
\eqref{equation:ttest1} / $3\beta $, 
\eqref{equation:ttest2} $\times$ $-1$, 
we get 
\begin{align}
&a\, \|q_x\|_{2}^2
-\frac{b}{2}\|q\|_{4}^2
+
\sqrt{-1}\, \int_X
q\bar{q}_x
dx, 
\label{equation:test1}
\\
&
3a\, 
\|q_{xx}\|_{2}^2
-
(aK+10b)
\int_X
|q|^2|q_x|^{2}
dx
-
(-aK+5b)
\operatorname{Re}
\int_X
q^2\bar{q}_x^2
dx.
\label{equation:test2}
\end{align}
In fact, 
via the relation \eqref{eq:frame0}, 
these quantities
\eqref{equation:test1}, \eqref{equation:test2} and 
$\|q\|_{2}^2$
are reformulated as $E_1(u)$, $E_2(u)$ and $\|u_x\|_{L^2}^2$
respectively. 
These quantities  
make sense and 
work effectively to prove Theorem~\ref{theorem:main} 
also when $X=\mathbb{T}$ or when the solution has no fixed point 
as $x\to -\infty$, 
as far as the K\"ahler manifold $N$ satisfies the condition 
$\nabla R\equiv 0$. 
Therefore, 
we can say that 
$\nabla J\equiv \nabla R\equiv 0$ is the assumption 
for the original equation \eqref{equation:pde} 
to behave essentially 
as a third order complex-valued 
nonlinear dispersive equation 
with constant coefficients, 
whose global existence result is well known.  
The proof of Theorem~\ref{theorem:main} 
itself will be given in the next section. 
\begin{remark} 
It seems to be reasonable to state the difference between 
our result and previous ones through the nonlinear structure of the 
equation \eqref{equation:t3rd}. 
It is known that 
the equation \eqref{equation:t3rd} is not necessarily 
completely integrable when $a\ne 0$ and $b\in \RR$, 
which is unlike the case for $a, b=0$.  
See, e.g., \cite{FM}, \cite{LZ}, \cite{ZS}. 
However, if $a\ne 0$ and $b=aK/2$, 
the equation \eqref{equation:t3rd} is so-called 
the Hirota equation
which is completely integrable. 
This is strongly related to the fact that 
there exists a conservation law to control 
$\nabla_x^2u_x(t)$ if
$N$ is a Riemann surface with constant 
curvature $K$ and $b=aK/2$. 
See \cite[Lemma~6.1]{onodera1}. 
\end{remark}
%
%
\section{Proof of the time-global existence theorem}
\label{section:proof}
First, we recall basic notation and facts to get estimation. 
We make use of basic techniques of geometric analysis of nonlinear problems. 
See \cite{nishikawa} for instance. 
For $u:X\to N$, $\Gamma(u^{-1}TN)$ denotes the set of the section
of $u^{-1}TN$, and 
$\lVert\cdot\rVert_{L^2}$ is a norm of $L^2(X;TN)$ defined by 
$$
\lVert{V}\rVert^2_{L^2}
=
\int_X
g(V,V)dx
\quad\text{for}\quad
V\in \Gamma(u^{-1}TN). 
$$
For positive integer  $k$, 
$H^{k+1}(X;N)$ denotes the set of 
all continuous mappings $u:X\to N$ 
satisfying $u_x\in H^k(X;TN)$, that is,  
$$
\lVert{u_x}\rVert_{H^{k}(X;TN)}^2
=
\sum_{l=0}^k 
\lVert{\nabla_x^lu_x}\rVert_{L^{2}}^2
=
\sum_{l=0}^k
\int_{X} 
g_{u(x)}(\nabla_x^lu_x(x),\nabla_x^lu_x(x))dx
<
+\infty.
$$
\par 
The main tools of the computation below are
\begin{align}
&\int_X
g(\nabla_xV,W)
dx
=-
\int_X
g(V,\nabla_xW)
dx,
\label{eq:ibp}
\\
  &\nabla_xu_t
 =
  \nabla_tu_x, 
\quad
\label{equation:commutator1}
\\
&\nabla_x^{k+1}u_t
=
\nabla_t\nabla_x^ku_x
+
\sum_{l=0}^{k-1}
\nabla_x^l
\left[ 
R(u_x, u_t)\nabla_x^{k-(l+1)}u_x 
\right], 
\quad 
k\in \mathbb{N}, 
\label{equation:commutator2}
\\
&R(V,W)=-R(W,V), \quad 
\text{in particular} \quad 
R(V,V)=0,
\label{eq:ten1} 
\\
&g(R(V_1,V_2)V_3,V_4)
=
g(R(V_3,V_4)V_1,V_2)
\label{eq:ten2}
\end{align}
for  $V, W, V_j\in\Gamma(u^{-1}TN)$, $j=1,2,3,4$, 
where $R$ is the Riemannian curvature tensor on $N$. 
In addition, the notation like 
$C$ or  $C(\cdot, \ldots, \cdot)$ 
will be sometimes used 
to denote a positive constant depending on certain parameters, 
such as $a$, $b$, 
geometric properties of N, et al. 
\vspace{0.3em}
\par
We start the proof of Theorem~\ref{theorem:main} 
from a short time existence result. 
Since the locally Hermitian symmetric space is the K\"ahler
manifold, short-time existence is ensured by 
the following: 
\begin{theorem}[Theorem~1.1 in \cite{onodera1} 
and Theorem~1.2 in \cite{onodera3}]
\label{theorem:localexistence}
Let $(N,J,g)$ be a compact K\"ahler manifold 
and let $a\ne 0$ and $b\in \RR$. 
Then for any $u_0{\in}H^{m+1}(X;N)$ 
with an integer $m\geqslant2$, 
there exists a constant $T>0$ 
depending only on $a$, $b$, $N$ and 
$\lVert{u_{0x}}\rVert_{H^2}$ 
such that the initial value problem 
\eqref{equation:pde}-\eqref{equation:data} 
possesses a unique solution 
$u{\in}C([-T,T];H^{m+1}(X;N))$. 
\end{theorem}
\par
Let $T$ be the largest number such that 
a solution $u(t,x)$ 
with the initial data $u_0\in H^{m+1}$  
exists on the interval $0\leqslant t<T$. 
If $\|u_x(t)\|_{H^m}$
is uniformly bounded on 
$[0,T)$, 
then we can extend the solution beyond $T$, 
which implies that the maximal existence time 
is infinite. 
Therefore,  
it suffices to show the following.
\begin{proposition}
\label{pro:apriori}
Let $u(t,x)$ be a solution of 
\eqref{equation:pde} with initial data 
$u_0\in H^{m+1}(X;N)$ on $[0,T)$, 
where $T$ is positive and finite number. 
Then $\|u_x(t)\|_{H^m}$ is uniformly bounded 
on $[0,T)$. 
\end{proposition}
\begin{proof}[Proof of Proposition~\ref{pro:apriori}]
We show the proof only for the case $X=\TT$, since  
the argument for the case $X=\RR$ is essentially parallel 
to the case $X=\TT$. 
We sometimes use Sobolev's inequality  
of the form
\begin{align}
\|V\|_{L^{\infty}}^2
&\leqslant
C
\|V\|_{L^2}
(\|V\|_{L^2}
+
\|\nabla_xV\|_{L^2})
\label{eq:so}
\end{align}
for $V\in \Gamma(u^{-1}TN)$ below with no mention. 
See, e.g., \cite[Lemma~1.~3. and 1.~4.]{koiso2} for the proof.
\par
Now, 
we establish two conservation laws and a 
semi-conservation law
on $[0,T)$  of the form 
\begin{align}
&\frac{d}{dt}\|u_x(t)\|_{L^2}^2=0,
\label{eq:conservation0}
\\
&
\frac{d}{dt}E_1(u(t))=0,
\label{eq:conservation1}
\\
&
\frac{d}{dt}E_2(u(t))=F(u(t)) 
\label{eq:conservation2}
\end{align}
for the solution $u(t,x)$, where 
\begin{equation}
|F(u(t))| 
\leqslant 
C(a,b,N, \|u_{x}(t)\|_{H^1})(1+\|\nabla_x^2u_x(t)\|_{L^2}^2).
\label{eq:F(u)}
\end{equation}
\par
Proposition~\ref{pro:apriori} is proved by  
\eqref{eq:conservation0}-\eqref{eq:F(u)} in 
the following manner: 
If \eqref{eq:conservation0} is true, then 
$\|u_x(t)\|_{L^2}=\|u_{0x}\|$ holds for $t\in [0,T)$. 
In addition, if \eqref{eq:conservation1} is true, 
by integrating \eqref{eq:conservation1} in $t$
and by using the inequality \eqref{eq:so}, 
we have 
\begin{align}
a\,\|\nabla_xu_x\|_{L^2}^2
&=
\frac{b}{2}
\int_X
\left( 
g(u_x,u_x)
\right)^2
dx
+\int_X
g(u_x, J\nabla_xu_x)
dx
+
E_1(u_0) 
\nonumber 
\\
&
\leqslant 
C_1(a,b,\|u_{0x}\|_{H^1})
+
C_2(b, \|u_{0x}\|_{L^2})
\left(
1+\|\nabla_xu_x\|_{L^2}
\right).
\nonumber 
\end{align} 
It means that  
$\|u_x(t)\|_{H^1}$ is uniformly bounded by some 
constant $C(a,b,\|u_{0x}\|_{H^1})$ on $[0,T)$. 
Thus if \eqref{eq:conservation2} and 
\eqref{eq:F(u)} are also true, 
after integrating \eqref{eq:conservation2} in $t$,
we get 
\begin{align}
a\, \|\nabla_x^2u_x(t)\|_{L^2}^2
&=
10b\, 
\int_X
\left( 
g(u_x,\nabla_xu_x)
\right)^2(t)
dx
+5b\, 
\int_X
g(u_x,u_x)
g(\nabla_xu_x, \nabla_xu_x)(t)
dx
\nonumber 
\\
&\quad 
-2a\, 
\int_X
g(R(u_x,\nabla_xu_x)u_x,\nabla_xu_x)(t)
dx
+
E_2(u_0)
+
\int_0^t
F(u(\tau))
d\tau
\nonumber 
\\
&\leqslant 
C_1(a,b,N,\|u_{0x}\|_{H^2})
+
C_2(a,b,N,\|u_{0x}\|_{H^1})
\int_0^t
(1+\|\nabla_x^2u_x(\tau)\|_{L^2}^2) 
d\tau.
\nonumber 
\end{align}
Therefore, the Gronwall lemma implies 
that 
$\|\nabla_x^2u_x(t)\|_{L^2}$ 
is uniformly bounded 
on $[0,T)$ and thus $\|u_x(t)\|_{H^2}$ 
is uniformly bounded  
on $[0,T)$.
Finally, the desired 
$H^m$-uniform estimate is obtained 
by using the estimate
\begin{align}
& \frac{d}{dt}
  \| u_x(t) \|_{H^k}^2
  \leqslant 
  C(a,b,N)
  P(
  \| u_x(t) \|_{H^{k-1}}
  )
  \| u_x(t) \|_{H^k}^2 
\label{eq:energyk}
\end{align}
inductively for $3\leqslant k\leqslant m$, 
where $P(\cdot)$ is some polynomial function on $\RR$.
The estimate \eqref{eq:energyk} 
has already been shown 
in \cite[Lemma~4.1]{onodera1} to prove the 
short-time existence theorem. 
\par 
From now on, 
we  
check \eqref{eq:conservation1}-\eqref{eq:F(u)}.
(First conservation law \eqref{eq:conservation0} 
is obvious, so we omit the computation.) 
We often use 
\eqref{eq:ibp}-\eqref{eq:so} with no mention below.
\par
To obtain \eqref{eq:conservation1}, we first deduce 
\begin{align}
\frac{d}{dt}
\left[
a\|\nabla_xu_x\|_{L^2}^2
\right] 
=&
2a\, 
\int_X
g(\nabla_xu_x, \nabla_t\nabla_xu_x)dx
\nonumber 
\\
=&
2a\, 
\int_X
g(\nabla_xu_x, \nabla_x^2u_t)dx
+
2a\, 
\int_X
g(\nabla_xu_x,  R(u_t, u_x)u_x)dx 
\nonumber 
\\
=&
2a\, 
\int_X
g(\nabla_x^3u_x, u_t)dx 
-
2a\, 
\int_X
g( R(u_x, \nabla_xu_x)u_x, u_t)dx.  
\label{eq:01}
\end{align} 
Since $u(t,x)$ solves the equation \eqref{equation:pde}, 
we have 
\begin{align}
2a\, 
\int_X
g(\nabla_x^3u_x, u_t)dx 
=&
2a\, 
\int_X
g(\nabla_x^3u_x, a\,\nabla_x^2u_x)dx
\nonumber 
\\ 
&+
2a\, 
\int_X
g(\nabla_x^3u_x, J\nabla_xu_x)dx 
\nonumber 
\\
&+
2a\, 
\int_X
g(\nabla_x^3u_x, b\,g(u_x, u_x)u_x)dx 
\nonumber 
\\
=& 
2ab\, 
\int_X
g(\nabla_x^3u_x, \,g(u_x, u_x)u_x)dx 
\nonumber 
\\
=& 
-4ab\, 
\int_X
g(\nabla_xu_x,u_x)g(u_x,\nabla_x^2u_x)
dx 
\nonumber 
\\
& 
-2ab\, 
\int_X
g(u_x,u_x)g(\nabla_xu_x,\nabla_x^2u_x)
dx 
\nonumber 
\\
=&
6ab\, 
\int_X
g(\nabla_xu_x,u_x)g(\nabla_xu_x,\nabla_xu_x)
dx, 
\label{eq:02} 
\\
-2a\, 
\int_X
g( R(u_x, \nabla_xu_x)u_x, u_t)dx
=&
-2a\, 
\int_X
g( R(u_x, \nabla_xu_x)u_x, a\nabla_x^2u_x)dx 
\nonumber 
\\
&-
2a\, 
\int_X
g( R(u_x, \nabla_xu_x)u_x, J\nabla_xu_x)dx 
\nonumber 
\\
&-
2a\, 
\int_X
g( R(u_x, \nabla_xu_x)u_x, b\,g(u_x,u_x)u_x)dx
\nonumber 
\\ 
=&
-2a^2\, 
\int_X
g( R(u_x, \nabla_xu_x)u_x, \nabla_x^2u_x)dx 
\nonumber 
\\
&-
2a\, 
\int_X
g( R(u_x, \nabla_xu_x)u_x, J\nabla_xu_x)dx 
\nonumber 
\\
=&
a^2\, 
\int_X
g((\nabla R)(u_x)(u_x, \nabla_xu_x)u_x, \nabla_xu_x)dx
\nonumber 
\\
&-
2a\, 
\int_X
g( R(u_x, \nabla_xu_x)u_x, J\nabla_xu_x)dx. 
\label{eq:03}
\end{align}
Remark that 
the second equality of \eqref{eq:03} 
follows from \eqref{eq:ten1} and 
the final equality of \eqref{eq:03} 
follows from \eqref{eq:ibp} 
and \eqref{eq:ten2}.
Substituting  \eqref{eq:02} and \eqref{eq:03} 
into \eqref{eq:01}, 
we obtain 
\begin{align}
\frac{d}{dt}
\left[
a\|\nabla_xu_x\|_{L^2}^2
\right] 
=&
6ab\, 
\int_X
g(\nabla_xu_x,u_x)g(\nabla_xu_x,\nabla_xu_x)
dx 
\nonumber 
\\
&+
a^2\, 
\int_X
g((\nabla R)(u_x)(u_x, \nabla_xu_x)u_x, \nabla_xu_x)dx
\nonumber 
\\
&-
2a\, 
\int_X
g( R(u_x, \nabla_xu_x)u_x, J\nabla_xu_x)dx. 
\label{eq:04}
\end{align}
In the same way, 
we deduce 
\begin{align}
&
\frac{d}{dt} 
\left[
-\frac{b}{2}
\int_X
\left(g(u_x,u_x)\right)^2
dx
\right]
\nonumber 
\\
&= 
-2b\, 
\int_X
g(u_x,u_x)g(u_x, \nabla_tu_x)
dx
\nonumber 
\\
&= 
-2b\, 
\int_X
g(u_x,u_x)g(u_x, \nabla_xu_t)
dx
\nonumber 
\\
&= 
4b\, 
\int_X
g(\nabla_xu_x,u_x)g(u_x,u_t)
dx
+2b\, 
\int_X
g(u_x,u_x)g(\nabla_xu_x,u_t)
dx
\nonumber 
\\
&=
4b\, 
\int_X
g(\nabla_xu_x,u_x)g(u_x, a\, \nabla_x^2u_x)
dx
\nonumber 
\\
&\quad 
+
4b\, 
\int_X
g(\nabla_xu_x,u_x)g(u_x,J\nabla_xu_x)
dx
\nonumber 
\\
&\quad 
+
4b\, 
\int_X
g(\nabla_xu_x,u_x)g(u_x,b\, g(u_x,u_x)u_x)
dx
\nonumber 
\\
&\quad 
+
2b\, 
\int_X
g(u_x,u_x)g(\nabla_xu_x, a\, \nabla_x^2u_x)
dx
\nonumber 
\\
&\quad 
+
2b\, 
\int_X
g(u_x,u_x)g(\nabla_xu_x,J\nabla_xu_x)
dx
\nonumber 
\\
&\quad 
+
2b\, 
\int_X
g(u_x,u_x)g(\nabla_xu_x,b\, g(u_x,u_x)u_x)
dx 
\nonumber 
\\
&=
4ab\, 
\int_X
g(\nabla_xu_x,u_x)g(u_x, \nabla_x^2u_x)
dx
\nonumber 
\\
&\quad 
+
4b\, 
\int_X
g(\nabla_xu_x,u_x)g(u_x,J\nabla_xu_x)
dx
\nonumber 
\\
&\quad 
+
2ab\, 
\int_X
g(u_x,u_x)g(\nabla_xu_x, \nabla_x^2u_x)
dx
\nonumber 
\\
&=
-6ab\, 
\int_X
g(\nabla_xu_x,u_x)g(\nabla_xu_x, \nabla_xu_x)
dx
\nonumber 
\\
&\quad 
+4b\, 
\int_X
g(\nabla_xu_x,u_x)g(u_x,J\nabla_xu_x)
dx. 
\label{eq:05}
\end{align}
Note that the final equality of \eqref{eq:05} 
comes from 
$$
\int_X
\left(
g(u_x,u_x)
\right)^2
g(u_x,\nabla_xu_x)
dx
=
\frac{1}{6}
\int_X
\left[
\left(
g(u_x,u_x)
\right)^3
\right]_x
dx
=0.
$$
Furthermore we deduce
\begin{align}
&\frac{d}{dt} 
\left[
-
\int_X
g(u_x, J\nabla_xu_x)
dx
\right]
\nonumber 
\\
&=
-\int_X
g(\nabla_tu_x, J\nabla_xu_x)
dx
-
\int_X
g(u_x, J\nabla_t\nabla_xu_x)
dx
\nonumber 
\\
&=
-\int_X
g(\nabla_xu_t, J\nabla_xu_x)
dx
-
\int_X
g(u_x, J\nabla_x^2u_t+JR(u_t,u_x)u_x)
dx
\nonumber 
\\
&=
2\int_X
g(u_t, J\nabla_x^2u_x)
dx
-
\int_X
g(R(u_x,Ju_x)u_x, u_t)
dx
\nonumber 
\\
&=
2\int_X
g(b\,g(u_x,u_x)u_x, J\nabla_x^2u_x)
dx
\nonumber 
\\
&\quad 
-
\int_X
g(R(u_x,Ju_x)u_x, a\,\nabla_x^2u_x)
dx
\nonumber 
\\
&\quad 
-
\int_X
g(R(u_x,Ju_x)u_x, J\nabla_xu_x)
dx.
\label{eq:06}
\end{align}
Here, for each term of right hand side of the above, 
a simple computation shows
\begin{align}
&
2\int_X
g(b\,g(u_x,u_x)u_x, J\nabla_x^2u_x)
dx
\nonumber 
\\
&\qquad \qquad 
=
-4b\,\int_X
g(\nabla_xu_x,u_x)g(u_x, J\nabla_xu_x)
dx, 
\label{eq:07}
\\
&-
\int_X
g(R(u_x,Ju_x)u_x, a\,\nabla_x^2u_x)
dx
\nonumber 
\\
&\qquad \qquad 
= 
a\, 
\int_X
g((\nabla R)(u_x)(u_x,Ju_x)u_x, \nabla_xu_x)
dx
\nonumber 
\\
&\qquad \qquad
\quad 
+
a\, 
\int_X
g(R(\nabla_xu_x,Ju_x)u_x, \nabla_xu_x)
dx
\nonumber 
\\
&\qquad \qquad 
\quad 
+
a\, 
\int_X
g(R(u_x,J\nabla_xu_x)u_x, \nabla_xu_x)
dx
\nonumber 
\\
&\qquad \qquad 
=
a\, 
\int_X
g((\nabla R)(u_x)(u_x,Ju_x)u_x, \nabla_xu_x)
dx
\nonumber 
\\
&\qquad \qquad 
\quad 
+
2a\, 
\int_X
g(R(u_x, \nabla_xu_x)u_x, J\nabla_xu_x)
dx, 
\label{eq:08}
\\
& 
-
\int_X
g(R(u_x,Ju_x)u_x, J\nabla_xu_x)
dx
\nonumber 
\\
&
\qquad \qquad 
=
-\frac{3}{4}
\int_X
g(R(u_x,Ju_x)u_x, J\nabla_xu_x)
dx
\nonumber 
\\
&\qquad \qquad 
\quad 
+
\frac{1}{4}
\int_X
g((\nabla R)(u_x)(u_x,Ju_x)u_x, Ju_x)
dx
\nonumber 
\\
&\qquad \qquad 
\quad 
+
\frac{1}{4}
\int_X
g(R(\nabla_xu_x,Ju_x)u_x, Ju_x)
dx
\nonumber 
\\
&\qquad \qquad 
\quad 
+
\frac{1}{4}
\int_X
g(R(u_x,J\nabla_xu_x)u_x, Ju_x)
dx
\nonumber 
\\
&\qquad \qquad 
\quad 
+
\frac{1}{4}
\int_X
g(R(u_x,Ju_x)\nabla_xu_x, Ju_x)
dx
\nonumber 
\\
&\qquad \qquad 
=
\frac{1}{4}
\int_X
g((\nabla R)(u_x)(u_x,Ju_x)u_x, Ju_x)
dx.  
\label{eq:09}
\end{align}
Substituting \eqref{eq:07}-\eqref{eq:09} into \eqref{eq:06}, 
we obtain 
\begin{align}
&\frac{d}{dt} 
\left[
-
\int_X
g(u_x, J\nabla_xu_x)
dx
\right]
\nonumber 
\\
&=
-4b\,\int_X
g(\nabla_xu_x,u_x)g(u_x, J\nabla_xu_x)
dx 
\nonumber 
\\
&\quad 
+
a\, 
\int_X
g((\nabla R)(u_x)(u_x,Ju_x)u_x, \nabla_xu_x)
dx
\nonumber 
\\
&\quad 
+
2a\, 
\int_X
g(R(u_x, \nabla_xu_x)u_x, J\nabla_xu_x)
dx
\nonumber 
\\
&\quad +
\frac{1}{4}
\int_X
g((\nabla R)(u_x)(u_x,Ju_x)u_x, Ju_x)
dx.   
\label{eq:10}
\end{align}
Consequently, by adding 
\eqref{eq:04}, \eqref{eq:05} and \eqref{eq:10}, 
we obtain 
\begin{align}
&\frac{d}{dt} 
\left[
a\, \|\nabla_xu_x\|_{L^2}^2
-\frac{b}{2}
\int_X
\left(g(u_x,u_x)\right)^2
dx
-
\int_X
g(u_x, J\nabla_xu_x)
dx
\right]
\nonumber 
\\
&=
a^2\, 
\int_X
g((\nabla R)(u_x)(u_x, \nabla_xu_x)u_x, \nabla_xu_x)dx
\nonumber 
\\
&\quad 
+
a\, 
\int_X
g((\nabla R)(u_x)(u_x,Ju_x)u_x, \nabla_xu_x)
dx
\nonumber 
\\
&\quad +
\frac{1}{4}
\int_X
g((\nabla R)(u_x)(u_x,Ju_x)u_x, Ju_x)
dx,    
\nonumber 
\end{align}
and the right hand side of the above 
vanishes due to the assumption 
$\nabla R\equiv 0$. 
Thus we obtain the conservation law \eqref{eq:conservation1}. 
\par 
We next show \eqref{eq:conservation2}.  
A simple computation gives  
\begin{align}
&\frac{d}{dt} 
\left[
3a\, \|\nabla_x^2u_x\|_{L^2}^2
\right]
\nonumber 
\\
&= 
6a\, 
\int_X
g(\nabla_t\nabla_x^2u_x, \nabla_x^2u_x)
dx
\nonumber 
\\
&=
6a\, 
\int_X
g(\nabla_x^3u_t+\nabla_x[R(u_t,u_x)u_x]
+R(u_t,u_x)\nabla_xu_x, \nabla_x^2u_x)
dx
\nonumber 
\\
&=
-6a\, 
\int_X
g(\nabla_x^5u_x, u_t)
dx
\nonumber 
\\
&\quad
+ 
6a\, 
\int_X
g(R(u_x,\nabla_x^3u_x)u_x, u_t)
dx
\nonumber 
\\
&\quad
- 
6a\, 
\int_X
g(R(\nabla_xu_x,\nabla_x^2u_x)u_x, u_t)
dx
\nonumber 
\\
&=
-6a\, 
\int_X
g(\nabla_x^5u_x, b\, g(u_x,u_x)u_x)
dx
\nonumber 
\\
&\quad
+ 
6a\, 
\int_X
g(R(u_x,\nabla_x^3u_x)u_x, a\, \nabla_x^2u_x)
dx
\nonumber 
\\
&\quad
+ 
6a\, 
\int_X
g(R(u_x,\nabla_x^3u_x)u_x, J\nabla_xu_x)
dx
\nonumber 
\\
&\quad
- 
6a\, 
\int_X
g(R(\nabla_xu_x,\nabla_x^2u_x)u_x, a\, \nabla_x^2u_x)
dx
\nonumber 
\\
&\quad
- 
6a\, 
\int_X
g(R(\nabla_xu_x,\nabla_x^2u_x)u_x, J\nabla_xu_x)
dx.
\label{eq:12}
\end{align}
Here, the integration by parts and 
the property of the Riemannian curvature tensor 
yield 
\begin{align}
\int_X
g(\nabla_x^5u_x,g(u_x,u_x)u_x)
dx
=&
-10 
\int_X
g(\nabla_x^2u_x,\nabla_xu_x)
g(\nabla_x^2u_x,u_x)
dx
\nonumber 
\\
&-5 
\int_X
g(\nabla_x^2u_x,\nabla_x^2u_x)
g(\nabla_xu_x,u_x)
dx, 
\label{eq:13}
\end{align}
\begin{align}
\int_X
g(R(u_x,\nabla_x^3u_x)u_x,\nabla_x^2u_x)
dx
=&
-\int_X
g(R(\nabla_xu_x,\nabla_x^2u_x)u_x,\nabla_x^2u_x)
dx 
\nonumber 
\\
&
-\frac{1}{2}
\int_X
g((\nabla R)(u_x)(u_x,\nabla_x^2u_x)u_x,\nabla_x^2u_x)
dx, 
\label{eq:14}
\end{align}
\begin{align}
\int_X
g(R(u_x,\nabla_x^3u_x)u_x,J\nabla_xu_x)
dx
=& 
-\int_X
g((\nabla R)(u_x)(u_x,\nabla_x^2u_x)u_x,J\nabla_xu_x)
dx
\nonumber 
\\
&
-\int_X
g(R(\nabla_xu_x,\nabla_x^2u_x)u_x,J\nabla_xu_x)
dx
\nonumber 
\\
&
-\int_X
g(R(u_x,\nabla_x^2u_x)\nabla_xu_x,J\nabla_xu_x)
dx
\nonumber 
\\
&
-\int_X
g(R(u_x,\nabla_x^2u_x)u_x,J\nabla_x^2u_x)
dx.
\label{eq:15}
\end{align}
By substituting \eqref{eq:13}-\eqref{eq:15} into 
\eqref{eq:12}, 
we obtain 
\begin{align}
\frac{d}{dt} 
\left[
3a\, \|\nabla_x^2u_x\|_{L^2}^2
\right]
&= 
-12a^2\, 
\int_X
g(R(\nabla_xu_x,\nabla_x^2u_x)u_x, \nabla_x^2u_x)
dx 
\nonumber 
\\
&\quad 
+60ab\, 
\int_X
g(\nabla_x^2u_x,\nabla_xu_x) 
g(\nabla_x^2u_x,u_x)
dx 
\nonumber 
\\ 
&\quad 
+30ab\, 
\int_X
g(\nabla_x^2u_x,\nabla_x^2u_x) 
g(\nabla_xu_x,u_x)
dx 
\nonumber 
\\ 
&\quad 
+F_0, 
\label{eq:16}
\end{align}
where 
\begin{align} 
F_0&=
-12a\, 
\int_X
g(R(\nabla_xu_x,\nabla_x^2u_x)u_x, J\nabla_xu_x)
dx 
\nonumber 
\\
&\quad 
-6a\, 
\int_X
g(R(u_x,\nabla_x^2u_x)\nabla_xu_x, J\nabla_xu_x)
dx 
\nonumber 
\\ 
&\quad 
-6a\, 
\int_X
g(R(u_x,\nabla_x^2u_x)u_x, J\nabla_x^2u_x)
dx 
\nonumber 
\\ 
&\quad 
-3a^2\,
\int_X
g((\nabla R)(u_x)(u_x,\nabla_x^2u_x)u_x,\nabla_x^2u_x)
dx 
\nonumber 
\\
&\quad 
-6a\, 
\int_X
g((\nabla R)(u_x)(u_x,\nabla_x^2u_x)u_x,J\nabla_xu_x)
dx. 
\label{eq:17}
\end{align}
Here, 
$F_0$ has the same estimate as \eqref{eq:F(u)}. 
To get the estimate, 
note that \eqref{eq:so} implies 
\begin{align}
&
\|\nabla_xu_x(t)\|_{L^{\infty}}
\leqslant 
C(\|u_x(t)\|_{H^{1}})
\left(
1+\|\nabla_xu_x(t)\|_{L^{2}}^2
\right)^{1/2}. 
\label{eq:eiji}
\end{align}
Then, 
it is easy to get 
\begin{align}
|F_0|
&\leqslant 
C(a,N)
\bigl\{
\|u_x \|_{L^{\infty}}
\|\nabla_xu_x \|_{L^{2}}
\|\nabla_xu_x \|_{L^{\infty}}
\|\nabla_x^2u_x \|_{L^{2}}
\nonumber 
\\
&
\qquad 
\qquad 
\qquad 
\qquad  
+
\left(
\|u_x \|_{L^{\infty}}^2+\|u_x \|_{L^{\infty}}^3
\right)
\|\nabla_x^2u_x \|_{L^{2}}^2
\nonumber 
\\
&
\qquad 
\qquad 
\qquad 
\qquad 
\qquad  
+
\|u_x \|_{L^{\infty}}^3
\|\nabla_xu_x \|_{L^{2}}
\|\nabla_x^2u_x \|_{L^{2}}
\bigr\}
\nonumber 
\\
&\leqslant
C(a,N,\|u_{x}\|_{H^1})(1+\|\nabla_x^2u_x\|_{L^2}^2).
\label{eq:18}
\end{align}
\par
To cancel the terms with higher order derivatives 
in the right hand side of \eqref{eq:16} 
except for $F_0$, 
we apply the rest part of the energy 
$E_2$. 
To neglect the effect of the lower order 
terms such as $F_0$, 
we use the notation
$f\equiv 0$ for any function $f(t)$ on $[0,T)$ 
if 
\begin{equation}
|f(t)|
\leqslant
C(a,b, N,\|u_{x}(t)\|_{H^1})(1+\|\nabla_x^2u_x(t)\|_{L^2}^2).
\label{eq:bounded}
\end{equation}
As we can see also from \eqref{eq:17}-\eqref{eq:18}, 
the integral where the sum of the order of the 
covariant derivative operator 
is less than five can be estimated as 
\eqref{eq:bounded}.
In other words, 
we have only to pay attention to 
the integral where 
the sum of the order of the covariant derivative 
is five. 
\par
Having them in mind, we first deduce 
\begin{align}
&\frac{d}{dt} 
\left[
-10b\,
\int_X
\left(
g(u_x,\nabla_xu_x)
\right)^2
dx
\right]
\nonumber 
\\
&=
-20b\,\int_X
g(u_x,\nabla_xu_x)
g(u_x,\nabla_t\nabla_xu_x)
dx
\nonumber 
\\
&\quad 
-
20b\,\int_X
g(u_x,\nabla_xu_x)
g(\nabla_tu_x,\nabla_xu_x)
dx
\nonumber 
\\
&=
-20b\,\int_X
g(u_x,\nabla_xu_x)
g(u_x,\nabla_x^2u_t)
dx
\nonumber 
\\
&\quad 
-20b\,\int_X
g(u_x,\nabla_xu_x)
g(u_x,R(u_t,u_x)u_x)
dx
\nonumber 
\\
&\quad 
-20b\,\int_X
g(u_x,\nabla_xu_x)
g(\nabla_xu_t,\nabla_xu_x)
dx
\nonumber 
\\
&=
20b\,\int_X
\left[g(u_x,\nabla_xu_x) \right]_x 
g(u_x,\nabla_xu_t)
dx
\nonumber 
\\
&=
-20b\,\int_X
g(u_x,\nabla_x^3u_x)
g(u_x,u_t)
dx
\nonumber 
\\
&\quad 
-
60b\,\int_X
g(\nabla_x^2u_x,\nabla_xu_x)
g(u_x,u_t)
dx
\nonumber 
\\
&\quad 
-
20b\,\int_X
g(\nabla_xu_x,\nabla_xu_x)
g(\nabla_xu_x,u_t)
dx
\nonumber 
\\
&\quad 
-
20b\,\int_X
g(u_x,\nabla_x^2u_x)
g(\nabla_xu_x,u_t)
dx
\nonumber 
\\
&
\equiv 
-
20b\,\int_X
g(u_x,\nabla_x^3u_x)
g(u_x,a\,\nabla_x^2u_x)
dx
\nonumber 
\\
&\quad 
-
60b\,\int_X
g(\nabla_x^2u_x,\nabla_xu_x)
g(u_x,a\,\nabla_x^2u_x)
dx
\nonumber 
\\
&\quad 
-
20b\,\int_X
g(\nabla_xu_x,\nabla_xu_x)
g(\nabla_xu_x,a\,\nabla_x^2u_x)
dx
\nonumber 
\\
&\quad 
-
20b\, \int_X
g(u_x,\nabla_x^2u_x)
g(\nabla_xu_x,a\,\nabla_x^2u_x)
dx
\nonumber 
\\
&=
-60ab\, \int_X
g(\nabla_xu_x,\nabla_x^2u_x)
g(u_x,\nabla_x^2u_x)
dx, 
\label{eq:24}
\end{align}
where the last equality follows from 
\begin{align}
\int_X
g(u_x,\nabla_x^3u_x)
g(u_x,\nabla_x^2u_x)
dx
=
-
\int_X
g(\nabla_xu_x,\nabla_x^2u_x)
g(u_x,\nabla_x^2u_x)
dx,
\label{eq:22}
\end{align}
\begin{align}
\int_X
g(\nabla_xu_x,\nabla_xu_x)
g(\nabla_xu_x,\nabla_x^2u_x)
dx
=
\frac{1}{4}
\int_X
\left[
\left(g(\nabla_xu_x,\nabla_xu_x)\right)^2
\right]_x
dx
=0.
\label{eq:23} 
\end{align}
Moreover, we deduce 
\begin{align}
&
\frac{d}{dt}
\left[
-5b\, 
\int_X
g(u_x,u_x)
g(\nabla_xu_x,\nabla_xu_x)
dx
\right]
\nonumber 
\\
=&
-10b\, 
\int_X
g(u_x,u_x)
g(\nabla_xu_x,\nabla_t\nabla_xu_x)
dx
\nonumber 
\\
&
-10b\, 
\int_X
g(\nabla_tu_x,u_x)
g(\nabla_xu_x,\nabla_xu_x)
dx
\nonumber 
\\
=&
-10b\, 
\int_X
g(u_x,u_x)
g(\nabla_xu_x,\nabla_x^2u_t)
dx
\nonumber 
\\
&
-10b\, 
\int_X
g(u_x,u_x)
g(\nabla_xu_x, R(u_t,u_x)u_x)
dx
\nonumber 
\\
&
-10b\, 
\int_X
g(\nabla_xu_t,u_x)
g(\nabla_xu_x,\nabla_xu_x)
dx
\nonumber 
\\
=&
-10b\, 
\int_X
g(u_x,u_x)
g(\nabla_x^3u_x,u_t)
dx
\nonumber 
\\
&
-40b\, 
\int_X
g(\nabla_xu_x,u_x)
g(\nabla_x^2u_x,u_t)
dx
\nonumber 
\\
&
-20b\, 
\int_X
g(\nabla_x^2u_x,u_x)
g(\nabla_xu_x,u_t)
dx
\nonumber 
\\
&
+10b\, 
\int_X
g(u_x,u_x)
g(R(u_x,\nabla_xu_x)u_x,u_t)
dx
\nonumber 
\\
&
-20b\, 
\int_X
g(\nabla_xu_x,\nabla_xu_x)
g(\nabla_xu_x,u_t)
dx
\nonumber 
\\
&
+20b\, 
\int_X
g(\nabla_x^2u_x,\nabla_xu_x)
g(u_x,u_t)
dx
\nonumber 
\\
&
+10b\, 
\int_X
g(\nabla_xu_x,\nabla_xu_x)
g(\nabla_xu_x,u_t)
dx
\nonumber 
\\
\equiv &
-10b\, 
\int_X
g(u_x,u_x)
g(\nabla_x^3u_x,a\,\nabla_x^2u_x)
dx
\nonumber 
\\
&
-40b\, 
\int_X
g(\nabla_xu_x,u_x)
g(\nabla_x^2u_x,a\,\nabla_x^2u_x)
dx
\nonumber 
\\
&
-20b\, 
\int_X
g(\nabla_x^2u_x,u_x)
g(\nabla_xu_x,a\,\nabla_x^2u_x)
dx
\nonumber 
\\
&
-10b\, 
\int_X
g(\nabla_xu_x,\nabla_xu_x)
g(\nabla_xu_x,a\,\nabla_x^2u_x)
dx
\nonumber 
\\
&
+20b\, 
\int_X
g(\nabla_x^2u_x,\nabla_xu_x)
g(u_x,a\,\nabla_x^2u_x)
dx 
\nonumber 
\\
=
&
-30ab\, 
\int_X
g(\nabla_xu_x,u_x)
g(\nabla_x^2u_x,\nabla_x^2u_x)
dx.
\label{eq:26}
\end{align}
Note that the last equality follows from 
\eqref{eq:23} and 
\begin{align}
\int_X
g(u_x,u_x)
g(\nabla_x^3u_x, \nabla_x^2u_x)
dx
=
-\int_X
g(\nabla_xu_x,u_x)
g(\nabla_x^2u_x, \nabla_x^2u_x)
dx. 
\nonumber 
\end{align}
In the same way, we get
\begin{align}
&\frac{d}{dt} 
\left[
2a\, 
\int_X
g(R(u_x,\nabla_xu_x)u_x,\nabla_xu_x)
dx
\right]
\nonumber 
\\
=&
4a\, 
\int_X
g(R(\nabla_tu_x,\nabla_xu_x)u_x,\nabla_xu_x)
dx
\nonumber 
\\
&+
4a\, 
\int_X
g(R(u_x,\nabla_t\nabla_xu_x)u_x,\nabla_xu_x)
dx
\nonumber 
\\
=& 
4a\, 
\int_X
g(R(\nabla_xu_t,\nabla_xu_x)u_x,\nabla_xu_x)
dx
\nonumber 
\\
&+
4a\, 
\int_X
g(R(u_x,\nabla_x^2u_t)u_x,\nabla_xu_x)
dx
\nonumber 
\\
&+
4a\, 
\int_X
g(R(u_x,R(u_t,u_x)u_x)u_x,\nabla_xu_x)
dx
\nonumber 
\\
=&
-4a\, 
\int_X
g(R(u_x,\nabla_xu_t)u_x,\nabla_x^2u_x)
dx
\nonumber 
\\
&
-8a\, 
\int_X
g(R(\nabla_xu_x,\nabla_xu_t)u_x,\nabla_xu_x)
dx
\nonumber 
\\
&
+4a\, 
\int_X
g(R(u_x,\nabla_xu_x)u_x,R(u_t,u_x)u_x)
dx
\nonumber 
\\
=&
4a\, 
\int_X
g(R(\nabla_xu_x,\nabla_x^2u_x)u_x,u_t)
dx
\nonumber 
\\
&+
4a\, 
\int_X
g(R(u_x,\nabla_x^3u_x)u_x,u_t)
dx
\nonumber 
\\
&+
12a\, 
\int_X
g(R(u_x,\nabla_x^2u_x)\nabla_xu_x,u_t)
dx
\nonumber 
\\
&+
8a\, 
\int_X
g(R(u_x,\nabla_xu_x)\nabla_x^2u_x,u_t)
dx
\nonumber 
\\
&
-4a\, 
\int_X
g(R(u_x, R(u_x,\nabla_xu_x)u_x)u_x,u_t)
dx
\nonumber 
\\
\equiv&
4a\, 
\int_X
g(R(\nabla_xu_x,\nabla_x^2u_x)u_x, a\,\nabla_x^2u_x)
dx
\nonumber 
\\
&+
4a\, 
\int_X
g(R(u_x,\nabla_x^3u_x)u_x,a\,\nabla_x^2u_x)
dx
\nonumber 
\\
&+
12a\, 
\int_X
g(R(u_x,\nabla_x^2u_x)\nabla_xu_x,a\,\nabla_x^2u_x)
dx
\nonumber 
\\
&+
8a\, 
\int_X
g(R(u_x,\nabla_xu_x)\nabla_x^2u_x,a\,\nabla_x^2u_x)
dx 
\label{eq:266}
\\
\equiv 
&
12a^2\, 
\int_X
g(R(u_x,\nabla_x^2u_x)\nabla_xu_x,\nabla_x^2u_x)
dx.
\label{eq:27}
\end{align}
Note that the last relation comes from the computation   
\begin{align}
&\int_X
g(R(u_x,\nabla_x^3u_x)u_x,\nabla_x^2u_x)
dx
\nonumber 
\\
&=
-\int_X
g(R(\nabla_xu_x,\nabla_x^2u_x)u_x,\nabla_x^2u_x)
dx
\nonumber 
\\
&\quad
-\frac{1}{2}
\int_X
g((\nabla R)(u_x)(u_x,\nabla_x^2u_x)u_x,\nabla_x^2u_x)
dx 
\nonumber 
\\
&\equiv
-\int_X
g(R(\nabla_xu_x,\nabla_x^2u_x)u_x,\nabla_x^2u_x)
dx, 
\nonumber 
\end{align}
and the fact that the last integral of the right hand side of 
\eqref{eq:266} vanishes because of \eqref{eq:ten1}.  
\\
As a consequence, 
if we add \eqref{eq:16}, \eqref{eq:24}, 
\eqref{eq:26} and  \eqref{eq:27}, 
we obtain 
$
(d/dt)
E_2(u)
\equiv 
F_0
\equiv 
0$, 
which implies desired  
\eqref{eq:conservation2} and \eqref{eq:F(u)}.
Thus we complete the proof.
\end{proof}
%
%
{\bf Acknowledgement.} \\
This work is supported by 
Global COE Program 
``Education and Research Hub for Mathematics-for-Industry''
and 
JSPS Grant-in-Aid for 
Young Scientists (B) \#21740101.


\begin{thebibliography}{99}
\bibitem{CSU}
Chang,~N.-H., Shatah,~J, Uhlenbeck,~K. \ (2000). \
{\it Schr\"odinger maps}, 
Comm.\ Pure Appl.\ Math.\ {\bf 53}, 590--602.

\bibitem{chihara} 
Chihara,~H. \ (2008) \
{\it Schr\"odinger flow into almost Hermitian manifolds}, 
submitted for publication, 
arXiv:0807.3395. 

\bibitem{CO}
Chihara,~H., Onodera,~E. \ (2009). \
{\it A third-order dispersive flow for closed curves into almost 
Hermitian manifolds}, 
J. Funct.\ Anal. {\bf 257}, 388-404.

\bibitem{darios}
Da Rios,~L.-S. \ (1906). \
{\it On the motion of an unbounded fluid with a vortex filament of any shape [in Italian]},
Rend.\ Circ.\ Mat.\ Palermo {\bf 22}, 117--135.

\bibitem{Ding}
Ding,~W.~Y. \ (2002). \
{\it On the Schr\"odinger flows}.
 Proceedings of the ICM,\ Vol.\ II., 283--291.

\bibitem{FM}
Fukumoto,~Y., Miyazaki,~T. \ (1991). \
{\it Three-dimensional distortions of a	vortex filament with axial
	velocity},
J.\ Fluid Mech.\ {\bf 222}, 369--416.

\bibitem{hasimoto}
Hasimoto,~H. \ (1972). \
{\it A soliton on a vortex filament},
J.\ Fluid.\ Mech.\ {\bf 51}, 477--485.

\bibitem{koiso}
Koiso,~N. \ (1997). \
{\it The vortex filament equation and 
     a semilinear Schr\"odinger equation in a Hermitian symmetric space},
Osaka J. Math.\ {\bf 34}, 199--214.

\bibitem{koiso2}
Koiso,~N. \ (1993). \
{\it Convergence to a geodesic},
 Osaka J.\ Math.\ {\bf 30}, 559--565. 

\bibitem{Laurey} 
Laurey,~C. \ (1997). \
{\it The Cauchy problem for a third order nonlinear Schr\"odinger equation},  
Nonlinear Anal.\ {\bf 29}, 121--158. 

\bibitem{LZ}
Li,~J., Zhang,~H-Q., Xu,~T., Zhang,~Y-X., Tian,~B. \ (2007). \
{\it 
Soliton-like solutions of a generalized variable-coefficient 
higher order nonlinear Schro\"dinger equation from 
inhomogeneous optical fibers with symbolic computation},
J.\ Phys.\ A {\bf 40}, 13299--13309. 

\bibitem{MCGAHAGAN}
McGahagan,~H. \  (2007). \
{\it An approximation scheme for Schr\"odinger maps}.
 Comm.\ Partial\ Differential\ Equations {\bf 32}, 375--400.

\bibitem{nishikawa}
Nishikawa,~S. \ (2002). \
``Variational Problems in Geometry'', 
Translations of Mathematical Monographs {\bf 205}, 
the American Mathematical Society.

\bibitem{NT}
Nishiyama,~T., Tani,~A.  \ (1996). \
{\it Initial and initial-boundary value problems for a vortex filament
	with or without axial flow},
SIAM J.\ Math.\ Anal.\  {\bf 27},1015--1023.

\bibitem{onodera1}
Onodera,~E.  \ (2008). \ 
{\it A third-order dispersive flow for closed curves into K\"ahler manifolds}, 
J. Geom.\ Anal. {\bf 18}, 889--918.

\bibitem{onodera2} 
Onodera,~E. \ (2008). \ 
{\it Generalized Hasimoto transform of one-dimensional dispersive flows into compact Riemann surfaces}, 
SIGMA Symmetry Integrability Geom.\ Methods Appl. 
{\bf 4}, article No. 044, 10 pages.

\bibitem{onodera3} 
Onodera,~E. \ (2008) \
{\it The initial value problem for a third-order dispersive flow into compact almost Hermitian manifolds}, 
submitted for publication, 
arXiv:0805.3219.

\bibitem{PWW}
Pang,~P.~Y.~H., Wang,~H.-Y., Wang,~Y.-D. \  (2002). \
{\it Schr\"odinger flow on Hermitian locally symmetric spaces}, 
Comm.\ Anal.\ Geom.\ {\bf 10}, 653--681.

\bibitem{SSB}
Sulem,~P.~-L., Sulem,~C., Bardos,~C. \ (1986). \
{\it On the continuous limit for a system of classical spins},
Comm.\ Math.\ Phys.\  {\bf 107}, 431--454.

\bibitem{TN}
Tani,~A., Nishiyama,~T.\ (1997). \
{\it Solvability of equations for motion of a vortex filament 
with or without axial flow},
Publ.\ Res.\ Inst.\ Math.\ Sci.\  {\bf 33}, 509--526.

\bibitem{ZS}
Zakharov,~V.~E., Shabat,~A.~B.\ (1972). \ 
{\it Exact theory of two-dimensional self-focusing
and one-dimensional self-modulation of waves in nonlinear media}.
 Soviet Phys.\ JETP.\ {\bf 34}, 62--69. 
\end{thebibliography}
\end{document}